\documentclass[12pt]{amsart}
\usepackage{amssymb,epsfig,amsmath,latexsym,amsthm}
\usepackage{graphicx}
\usepackage{enumitem}
\usepackage{comment}
\usepackage[hidelinks]{hyperref}
\theoremstyle{plain}
\newtheorem{theorem}{Theorem}[section]
\newtheorem{lemma}[theorem]{Lemma}

\theoremstyle{definition}
\newtheorem{definition}[theorem]{Definition}
\newtheorem{remark}[theorem]{Remark}

\usepackage{epsfig,color}

\makeatother

\theoremstyle{definition}

\def\fnum{equation}

\numberwithin{equation}{section}

\usepackage[hidelinks]{hyperref}

\begin{document}
\title[A curvature gap for  minimal surfaces in the ball]
{A remark on a curvature gap for minimal surfaces in the  ball}

\author{Ezequiel Barbosa}
\address{Universidade Federal de Minas Gerais (UFMG), Caixa Postal 702, 30123-970, Belo Horizonte, MG, Brazil}
\email{ezequiel@mat.ufmg.br}
\author{Celso Viana}
\address{Department of Mathematics, University College London, Gower Street,
	London WC1E 6BT, United Kingdom} 
\email{celso.viana.14@ucl.ac.uk}

\begin{abstract}
 We extend to   higher codimension  earlier characterization of the  equatorial disk and  the critical  catenoid    by a pinching condition on the length of their second fundamental form among free boundary minimal surfaces in the  three dimensional Euclidean ball due to L. Ambrozio and I. Nunes.
\end{abstract}

\maketitle

\section{Introduction}\label{intro}

In this note we consider $2$-dimensional free boundary minimal surfaces in the Euclidean ball $B^n$. The free boundary condition implies that these   minimal surfaces  meet  the boundary of the ball  orthogonally. Such surfaces arise as  critical points of the area functional for relative cycles in the ball and  as extremals for the Steklov  problem \cite{FS1,FS2}.  The simplest free boundary minimal surfaces in the ball are  the equatorial disk and the critical catenoid.
Recently, Ambrozio and Nunes  \cite{AN} proved a characterization of the equatorial  disk and the critical catenoid in the $3$-dimensional ball  by a pinching condition involving the length of the second fundamental form and the support function:

\begin{theorem}[Ambrozio-Nunes]\label{an}
	\textit{Let $\Sigma^2$ be a compact free boundary minimal surface in $B^3$. Assume that for all points $x \in \Sigma$,
		\begin{eqnarray}\label{gap}
		|x^{\perp}|^2\, |A(x)|^2\leq 2
		\end{eqnarray}
		where $x^{\perp}$ denotes the normal component of the point $x\in \Sigma$ and $A$
		denotes the second fundamental form of $\Sigma$. Then
		\begin{enumerate}
			\item $|x^{\perp}|^2\,|A(x)|^2\equiv 0$  and $\Sigma$ is an  equatorial disk; 
			\item $|x_0^{\perp}|^2\,|A(x_0)|^2=2$ at some point $x_0\in \Sigma$ and $\Sigma$ is a critical catenoid.
	\end{enumerate}}
\end{theorem}

Theorem \ref{an}  share  similarities with  a classical theorem
of Chern, do Carmo, and Kobayashi  \cite{CCK} (see also Lawson \cite{L}) which characterizes  the equatorial spheres and  the Clifford hypersurfaces in $\mathbb{S}^{n+1}$ and the Veronese surface in $\mathbb{S}^4$    as the only  minimal submanifolds of dimension $n$  in  $\mathbb{S}^{n+p}$ satisfying the inequality $|A|^2\leq n\bigg/\bigg(2- \frac{1}{p}\bigg)$.
Despite the analogy, the proof of Theorem \ref{an} given in \cite{AN} is quite different. Besides working  in  dimension two, the codimension one is crucially used in some steps of the proof. The authors in \cite{AN} ask the question weather  Theorem \ref{an} can be generalized to  higher ambient dimension and submanifold co-dimension. 
In this note we answer their question positively in the special case of $2$-dimensional surfaces in the  ball of any dimension, see Theorem \ref{our result} below.

Our proof   follows closely the arguments in \cite{AN} and it is  based on three  ingredients: Fraser and Schoen's Uniqueness  Theorem for free boundary minimal disks, a standard dimension reduction argument, and an analysis of nodal sets for  solutions of an elliptic system of    partial differential equations.

We remark that Theorem \ref{an} was    recently generalized to  geodesic balls in the $3$-dimensional hyperbolic space and the hemisphere in \cite{HS}. Our proof   also applies  to these settings and their result can   be extended in a similar way as discussed here.   Finally, we mention that  the pinching condition (\ref{gap})  also characterizes the  plane and the catenoid among properly embedded  minimal surfaces without boundary  in $\mathbb{R}^3$ (see Remark \ref{remark} below). 
A  version of this result  was  first proved by   Meeks, P\'{e}rez, and Ros in \cite[Section 7]{MPR}.

\section{Preliminaries}

The next two lemmas are standard, for the benefit of the reader we include their proofs.

\begin{lemma} \label{mean curvature equation}
	\textit{Let $\Sigma^k$ be a minimal submanifold in  $\mathbb{R}^{n}$ given by the graph of the function $u:U\subset \mathbb{R}^k\rightarrow \mathbb{R}^{n-k}$ where
		$u(x)=(u_1(x),\ldots,u_{n-k}(x))$. Then for every $l=1,\ldots,n-k$
		\begin{eqnarray}\label{mean curvature}
		\frac{a_{ij}(\nabla u_1,\ldots,\nabla u_{n-k})}{\sqrt{1+|\nabla u_l|^2}}\,D_{ij} u_l=0,
		\end{eqnarray}
		for some smooth functions $a_{ij}(\nabla u_1,\ldots,\nabla u_{n-k})$.}
\end{lemma}
\begin{proof}
	Parametrize $\Sigma$  as $\varphi(x)=(x,u_1(x),\ldots,u_{n-k}(x))$. The coordinate basis for $\Sigma$ is given by \[D_{x_i}\varphi= (0,\ldots,1,\ldots, D_{x_i}u_1,\ldots, D_{x_i}u_{n-k}),\]
	for $i\,=\,1,\ldots,k$.  It follows that
	\begin{equation}\label{metrica inversa}
	g_{ij}=\delta_{ij}+ \sum_{l=1}^{n-k}D_{x_i}u_l\,D_{x_j}u_l\quad \text{and}\quad g^{ij}=a_{ij}(\nabla u_1,\ldots,\nabla u_{n-k}).
	\end{equation}
	Now we consider for each $l\,=\, 1,\ldots, n-k$ the unit normal vector 
	\[N_l=\frac{1}{\sqrt{1 + |\nabla\,u_l|^2}}(-D_{x_1}u_l,\ldots, -D_{x_k}u_l,0,\ldots, 1,\ldots, 0).\]
	A simple computation gives
	\begin{eqnarray*}(N_l)_{x_i}&=&\bigg(\frac{1}{\sqrt{1 + |\nabla\,u_l|^2}}\bigg)_{x_i}\sqrt{1+ |\nabla u_l|^2}\,N_l+ \\
		&&\frac{1}{\sqrt{1 + |\nabla\,u_l|^2}} (-D_{x_1x_l}^2u_l,\ldots, -D_{x_k x_l}^2u_l,0\ldots,0).
	\end{eqnarray*}
	Consequently,
	\[(A_{N_l})_{ij}=\langle -d N_l(\varphi_{x_i}),\varphi_{x_j}\rangle=\frac{1}{\sqrt{1 + |\nabla\,u_l|^2}}D_{x_i x_j}^2u_l. \]
	Since  $\Sigma^k$ is minimal, $0=g^{ij}(A_{N_l})_{ij}$ and by (\ref{metrica inversa})  we obtain
	\[\frac{a_{ij}(\nabla u_1,\ldots,\nabla u_{n-k})}{\sqrt{1 + |\nabla\,u_l|^2}} D_{ij} u_l\,=\,0.\]
\end{proof}
\begin{lemma}\label{difference equation}
	\textit{If $u,v: U\subset \mathbb{R}^k\rightarrow \mathbb{R}^p$ are smooth maps  satisfying (\ref{mean curvature}), then the difference $\varphi=u-v$ satisfies, for each $l=1,\ldots,p$, the  equation
		\begin{eqnarray*}
			\frac{a_{ij}(\nabla u)}{\sqrt{1+ |\nabla u_l|^2}}D_{ij}(\varphi_l) + \sum_{m=1}^{p}b_j^m(\nabla u,\nabla v) D_{j} (\varphi_m)=0,
		\end{eqnarray*}
	 for some smooth functions $a_{ij}(\nabla u)$ and $b_j^m(\nabla u,\nabla v)$.}
\end{lemma}
\begin{proof}
	As $u_l$ and $v_l$ satisfy equation (\ref{mean curvature}), therefore
	\begin{eqnarray*}
		0=\frac{a_{ij}(\nabla u)}{\sqrt{1+ |\nabla u_l|^2}}D_{ij}u_l &-& \frac{a_{ij}(\nabla v)}{\sqrt{1+ |\nabla v_l|^2}}D_{ij}v_l  \\
		= \frac{a_{ij}(\nabla u)}{\sqrt{1+ |\nabla u_l|^2}}D_{ij}u_l &-&  \frac{a_{ij}(\nabla u)}{\sqrt{1+ |\nabla u_l|^2}}D_{ij}v_l \\ 
		&+& \frac{a_{ij}(\nabla u)}{\sqrt{1+ |\nabla u_l|^2}}D_{ij}v_l - \frac{a_{ij}(\nabla v)}{\sqrt{1+ |\nabla v_l|^2}}D_{ij}v_l \\
		=\frac{a_{ij}(\nabla u_l)}{\sqrt{1+ |\nabla u_l|^2}}D_{ij}(\varphi_l) &+& \bigg(\frac{a_{ij}(\nabla u)}{\sqrt{1+ |\nabla u_l|^2}} - \frac{a_{ij}(\nabla v)}{\sqrt{1+ |\nabla v_l|^2}}\bigg)D_{ij}v_l.
	\end{eqnarray*}
	Now, let $F_{ij}: \mathbb{R}^k\times\cdots\times \mathbb{R}^k\rightarrow \mathbb{R}$ be the function  defined by \[F_{ij}(z_1,\ldots,z_{p})= \frac{a_{ij}(z_1,\ldots,z_{p})}{\sqrt{1+ |z_l|^2}}.\] By the Fundamental Theorem of Calculus we can write
	\begin{eqnarray*}
		F_{ij}(\nabla u)- F_{ij}(\nabla v)=\bigg(\int_{0}^{1}dF_{ij}(\nabla u + t(\nabla v - \nabla u))dt\bigg)\nabla (u-v).
	\end{eqnarray*}
	The lemma follows by setting $b_q^m$ to be
	\begin{eqnarray*}
		b_q^m=\bigg(\int_{0}^{1}dF_{ij}(\nabla u + t(\nabla v - \nabla u))dt\bigg)_{qm}\, D_{ij}v_l\,D_q (u-v)_m.
	\end{eqnarray*}
\end{proof}

The next lemma, which  is essentially contained in   \cite{HS}, concerns nodal sets for solutions of  elliptic   equations. We add the proof in order to include  solutions of   elliptic  system of   equations.
\begin{lemma}[Hardt-Simon \cite{HS}]\label{size nodal set}
	\textit{Let $u: U\subset \mathbb{R}^n\rightarrow \mathbb{R}^p$ be a smooth map which satisfies for each $k=1,\ldots,p$ an elliptic equation of the form:
		\begin{eqnarray}\label{pde}
		a_{ij}(x)D_{ij} u_k + \sum_{l=1}^{p}b_j^l(x) D_j u_l+ \sum_{l=1}^{p}c_l(x)\,u_l=0,
		\end{eqnarray}
		where $a_{ij}$, $b_j$, and  $c_l$ are smooth functions. Let's assume that $a_{ij}$ is positive definite, and that $|b_j|\leq C$ and $|c_l|\leq C$ for some constant $C>0$. If the order of vanishing of $u_l$ at $u_l^{-1}(0)$ is finite for each $l$ and if $x_0 \in u^{-1}(0)\cap |Du|^{-1}(0)$, then 
		\[ u^{-1}(0)\cap |Du|^{-1}(0)\cap B_r(x_0)\] decomposes into a countable union of subsets of a pairwise disjoint collection of $n-2$-dimensional smooth submanifolds.}
\end{lemma}
\begin{proof}
	We  define for each integer $q=1,2,\ldots$ the set
	\begin{equation}
	S_q=\{x:D^{\alpha}u_l(x)=0,\, \forall\, |\alpha|\leq q,\,\, \forall l\,\, \text{and}\,\, D^{q +1}u_{l_0}(x)\neq 0\,\, \text{for some}\,\,l_0 \}.
	\end{equation}
	We first note that if $x \in  u^{-1}(0)\cap|Du|^{-1}(0)$ and $r>0$ is small enough, then 
	\begin{eqnarray}\label{zero set}
	u^{-1}(0)\cap |Du|^{-1}(0)\cap B_r(x)
	= \cup_{q=1}^{d} S_q \cap B_r(x),
	\end{eqnarray}
	where $d-1$ is the order of vanishing of $u$ at $x$.
	Now for each $x \in S_q$ we consider a multi-index $\beta$ such that $|\beta|= q-1$ and  $\text{Hess}(D^{\beta}u_{l_0})(x)\neq 0$ for some $l_0$. Applying $D^{\beta}$ to both sides of (\ref{pde})  with $k=l_0$ and recalling that $D^{\alpha}u_{l}(x)=0$ for every multi-index $\alpha$ such that $|\alpha|\leq q$ we obtain
	\[a_{ij}(x)D_{ij}( D^{\beta}u_{l_0})(x)=0.\]
	Using that $a_{ij}$ is positive definite and that $\text{Hess}(D^{\beta}u_{l_0})(x)\neq 0$ we conclude that $\text{rank}(\text{Hess}(D^{\beta}u_{l_0})(x)\geq 2$. Thus there exist indexes $i_1$ and $i_2$ for which $\text{grad}(D_{i_1}D^{\beta}u_{l_0})(x)$ and $\text{grad}(D_{i_2}D^{\beta}u_{l_0})(x)$ are linearly independent. This implies that for small $r>0$ that
	\[B_r(x)\cap (D_{i_1}D^{\beta}u_{l_0})^{-1}(0)\cap (D_{i_2}D^{\beta}u_{l_0})^{-1}(0) \]
	is a $n-2$-dimensional submanifold $\Sigma_{x,r,\beta}$ which contains $B_r(x)\cap S_q$.
	In view of (\ref{zero set}) we conclude that for each $x\in u^{-1}(0)\cap |Du|^{-1}(0)$ there exist $r>0$ and smooth $n-2$-dimensional submanifolds $\Sigma_{x,r,q_1},\ldots, \Sigma_{r,x,q_s}$ for which
	\begin{eqnarray}\label{n-2}
	B_r(x)\cap u^{-1}(0)\cap|Du|^{-1}(0)\subset \cup_{j=1}^s \Sigma_{x,r,q_j}. 
	\end{eqnarray}
	The Lemma follows from (\ref{n-2}).
\end{proof}
\begin{lemma}\label{unique continuation}
	\textit{If $\Sigma_1$ and $\Sigma_2$ are $2$-dimensional minimal surfaces in $\mathbb{R}^n$ having a tangential intersection of infinite order at $x_0\in\Sigma_1\cap\Sigma_2$, then $\Sigma_1=\Sigma_2$.}
\end{lemma}
\begin{proof}
	Let  $v_k: \Omega \rightarrow \mathbb{R}^n$ be minimal map parameterizing a neighborhood of  $\Sigma_k$ for each $k=1,2$ and assume that $v_k(0)=x_0$.
	We can assume that the coordinates $z=x+y\, i$ in $\Omega$ are isothermal for both $v_1$ and $v_2$. As $v_k$ is minimal, each coordinate $v_k^i$, $i=1,\ldots, n$, is harmonic, which implies by the conformal invariance  of the Laplacian that
	$
	\partial_{\overline{z}}\partial_z v_k^i=0
	$. Hence, if we define $v(z)=v_1(z)- v_2(z)$, then each component of $\partial_{z} v$ is holomorphic, i.e., $\partial_{\overline{z}}\partial_zv^i=0$. Since $z=0$ is an infinite order zero of $v$, the analytic continuation property for holomorphic functions implies that $v\equiv 0$. Therefore, $\Sigma_1=\Sigma_2$.
\end{proof}

\section{Proof of  theorem}

\begin{theorem}\label{our result}
	\textit{Let $\Sigma^2$ be a compact free boundary minimal surface in $B^n$. Assume that for all points $x \in \Sigma$,
		\begin{equation}\label{pinching}
		|x^{\perp}|^2|A(x)|^2\leq 2,
		\end{equation}
		where $x^{\perp}$ denotes the normal component of $x$. Then
		\begin{enumerate}
			\item $|x^{\perp}|^2|A(x)|^2\equiv 0$ and $\Sigma^2$ is a flat equatorial disk.
			\item  $|x_0^{\perp}|^2|A(x_0)|^2=2$ at some point $x_0\in\Sigma^2$  and  $\Sigma^2$ is the critical catenoid inside a $3$-dimensional linear subspace.
	\end{enumerate}}
\end{theorem}

\begin{lemma}\label{hessian}
	\textit{Let $\Sigma^2$ be a free boundary minimal surface in $B^n$ and $f$ be the function  $f:\Sigma^2 \rightarrow \mathbb{R}$ defined by  $$f(x)=\frac{|x|^2}{2},\,\, x \in \Sigma^2.$$ Then
		$\nabla^{\Sigma} f= x^{\top}$ for every $x \in  \Sigma$ and
		\begin{eqnarray}\label{hessiana formula}
		\text{Hess}_{\Sigma}\,f(x)(X,Y)=\langle X,Y\rangle + \langle A(X,Y),\overrightarrow{x}\rangle.
		\end{eqnarray}}
\end{lemma}
\begin{proof}
	Given $X\in \mathcal{X}(\Sigma)$, then
	\[
	X(f)= \frac{1}{2} X \langle \overrightarrow{x},\overrightarrow{x}\rangle= \langle X, \overrightarrow{x}\rangle= \langle X, x^{\top}\rangle.
	\]
	Hence,  $\nabla^{\Sigma} f(x)= x^{\top}$. The hessian of $f$ is then given by
	\begin{eqnarray*}
		Hess_{\Sigma}\,f\,(X,Y)&=& \langle \nabla_X\nabla f,Y\rangle= \langle \overline{\nabla}_X \nabla f,Y\rangle = \langle \overline{\nabla}_X (x- x^{\perp}),Y\rangle \\
		&=& \langle X,Y\rangle -\langle \overline{\nabla}_X x^{\perp},Y\rangle= \langle X,Y\rangle + \langle x^{\perp},\overline{\nabla}_XY\rangle \\
		&=& \langle X,Y\rangle + \langle A(X,Y),\overrightarrow{x}\rangle,
	\end{eqnarray*}
	where $X$ and $Y$ are vector fields in $\mathcal{X}(\Sigma)$.
\end{proof}
\begin{lemma}\label{hessiana positiva}
	If $|x^{\perp}|^2|A(x)|^2\leq 2$, then $\text{Hess}_{\Sigma}f\geq0$.
\end{lemma}
\begin{proof}
	Let $\{e_1,e_2\}$ be an orthonormal base of $T\Sigma$ given by eigenvectors of $\text{Hess}_{\Sigma}f$. The respective eigenvalues are $\overline{\lambda}_i= 1 + \langle A(e_i,e_i),\overrightarrow{x}\rangle$. We want to prove that $\overline{\lambda}_i\geq 0$ for $i=1,2$.
	\begin{equation}\label{cauchy-schwarz}
	\overline{\lambda}_1^2 + \overline{\lambda}_2^2= 2+ \sum_{i=1}^{2}\left< A(e_i,e_i),\overrightarrow{x}\right>^2\leq 2 + \sum_{i=1}^{2}|A(e_i,e_i)|^2x^{\perp}|^2\leq 2 + |A|^2|x^{\perp}|^2,
	\end{equation}
where we used the Cauchy-Schwarz inequality in the first inequality.  Since$(\overline{\lambda}_1+ \overline{\lambda}_2)^2= 4$, we conclude that 
\[
2\,\overline{\lambda}_1\,\overline{\lambda}_2\geq2 - |A|^2\,|x^{\perp}|^2\geq 0.
\]
Hence, $\overline{\lambda}_1$ and $\overline{\lambda}_2$ have the same sign. As $\overline{\lambda}_1+ \overline{\lambda}_2= 2$, the lemma is proved.
\end{proof}
\begin{definition}
	Given a $2$-dimensional free boundary minimal surface $\Sigma^2$ in $B^n$ we define
	\begin{eqnarray}\label{C}
	C(\Sigma)= \{x\in \Sigma: f(x)\,=\,m_0\,:=\, \min_{\Sigma}\,f\}.
	\end{eqnarray}
\end{definition}

The conormal vector of a  free boundary minimal surface $\Sigma$ being normal to the boundary of the ball implies that $\partial\Sigma$ is convex on $\Sigma$. Using this fact and that $\text{Hess}_{\Sigma}f\geq 0$, we  obtain:

\begin{lemma}\label{total convexity}
	If $\Sigma$ is a free boundary minimal surface in $B^n$ satisfying $\text{Hess}_{\Sigma}f\geq 0$, then the set $C(\Sigma)$ is totally convex on $\Sigma$, meaning that every geodesic segment with extremities in $C(\Sigma)$ is in $C(\Sigma)$.
\end{lemma}

Before  we start proving  Theorem \ref{our result}, let us recall a simple fact from Riemannian Geometry that we will use later. Let $c:[a,b]\rightarrow M$ be a curve in a Riemannian manifold $M$ and 
$P_s:T_{c(a)}M\rightarrow T_{c(s)}M$    the parallel transport map along $c$. Let $\Delta$ be a correspondence   which associates for each $s\in [a,b]$ a $j$-dimensional subspace $\Delta(s)\subset T_{c(s)}M$. The distribution $\Delta(s)$ is called parallel if $P_s(\Delta(a))= \Delta(s)$ for every $s \in [a,b]$.
\begin{lemma}[Spivak \cite{S}]\label{parallel lema}
	\textit{If  $\frac{DV}{ds}(s)\in \Delta(s)$ whenever $V$ is a vector field in $\Delta(s)$, then $\Delta(s)$ is parallel along $c$.}
\end{lemma}

\begin{proof}[Proof of Theorem \ref{our result}]
	By Lemma \ref{hessiana positiva}, the inequality (\ref{pinching}) implies that $\text{Hess}_{\Sigma}\,f \geq 0$.
	Let us show  this implies that $\Sigma$ is diffeomorphic to either a disk or an annulus.

	If $\Sigma$ is simply connected, then $\Sigma$ is topologically a disk. Hence, we  assume that $\pi_1(\Sigma, x)\neq \{0\}$, where $x$ is chosen to lie in $C(\Sigma)$. By minimizing the length in a nontrivial homotopy class $[\alpha]\in \pi_1(\Sigma,x)$ among closed loops passing through the fixed point $x\in C(\Sigma)$,  we obtain a geodesic loop $\gamma: [0,1]\rightarrow \Sigma$, where $\gamma(0)=\gamma(1)=x$; this follows from the fact that  $\partial \Sigma$  is convex on $\Sigma$ due to the free boundary condition. We claim that $\gamma^{\prime}(0)=\gamma^{\prime}(1)$ and $C(\Sigma)=\gamma([0,1])$. If either one of those properties are not true, then the total convexity of $C(\Sigma)$ guarantees that an open set $U$ of $\Sigma$ is contained in $C(\Sigma)$. In this case, $\text{Hess}_{\Sigma}f\equiv 0$ over $U$. Hence, 
	\[\langle A(X,Y),\overrightarrow{x}\rangle = -\,\langle X,Y\rangle.
	\]
Since $\overrightarrow{x}$ is a constant length normal vector  to $\Sigma$ along $U$, we conclude that the mean curvature of $\Sigma$ in the direction of $\overrightarrow{x}$ is non-zero, a contradiction. 
Therefore, $C(\Sigma)$ is a smooth simple closed geodesic. Note that this implies that $\pi_1(\Sigma)$ is cyclic, from this we obtain that  $\Sigma$ is an annulus. 
	
	If $\Sigma^2$ is a minimal disk, then Fraser and Schoen's theorem in \cite{FS} implies that $\Sigma^2$ is an equatorial disk and   $|x^{\perp}|^2\,|A(x)|^2\equiv 0$.

	If $\Sigma^2$ is an annulus, then $C(\Sigma)$ is a  smooth simple closed geodesic. This implies that   $\overline{\lambda}=0$ is an eigenvalue of $\text{Hess}_{\Sigma}\,f(x_0)$ for every $x_0\in C(\Sigma)$ and $(\overline{\lambda}_1^2+ \overline{\lambda}_2^2)=(\overline{\lambda}_1+\overline{\lambda}_2)^2$. On the other hand, 
	\[
	\sum_{i=1} \overline{\lambda}_i^2= 2 + \langle A(e_i,e_i),\overrightarrow{x}\rangle^2\leq  2+ |x^{\perp}|^2\,|A(x)|^2\leq 4= \bigg(\sum_{i=1}\overline{\lambda}_i\bigg)^2.
	\] 
Hence, $\langle \sum_{i=1}^{2}A(e_i,e_i),\overrightarrow{x}\rangle^2= |A(x)|^2|x^{\perp}|^2=2$ for every $x\in C(\Sigma)$. It follows from   the Cauchy-Schwarz inequality  that
	\begin{eqnarray}\label{quase umbilica}
	A(e_i,e_i)=\langle A(e_i,e_i),\frac{\overrightarrow{x}}{|x|}\rangle \frac{\overrightarrow{x}}{|x|}.
	\end{eqnarray}
	Consequently, if $e_1$ is tangent to $C(\Sigma)$, then 
	\[
	\overline{\nabla}_{e_1}e_1=\langle A(e_i,e_i),\frac{\overrightarrow{x}}{|x|}\rangle \frac{\overrightarrow{x}}{|x|},
	\]
	since $C(\Sigma)$ is a geodesic on $\Sigma$. Thus,
	$C(\Sigma)$ is also a geodesic in $\partial B_{2\,m_0}^{n+1}(0)$, i.e., a round circle.
	Now  we consider the  normal distribution $E$ along $C(\Sigma)$ defined by \[E=\{\xi: \xi\in \mathcal{X}^{\perp}(\Sigma)|_{C(\Sigma)} \quad \text{and}\quad \langle \xi,\overrightarrow{x}\rangle=0\}.\]
	It follows from (\ref{quase umbilica}) that  for every $\xi \in E$ the following is true:
	\[\overline{\nabla}_{\gamma'(t)}\xi \in E.\]
	Lemma \ref{parallel lema} implies that  the distribution $E$ is parallel along $C(\Sigma)$. Hence, $E$ is a constant $(n-2)$-dimensional plane throught the origin.
	Therefore, there exists a critical catenoid $\Sigma_c$  which is tangent to $\Sigma$ along $C(\Sigma)$. Near $x_0\in C(\Sigma)$ we write $\Sigma$ and $\Sigma_c$ locally as a graph over $T_{x_0}\Sigma$. Hence, $\Sigma_c=\text{graph}(f_c)$ and
	\[ \text{div}\bigg(\frac{\nabla f_c}{\sqrt{1+ |\nabla f_c}|^2}\bigg)=0.\]
	Similarly, $\Sigma=\text{graph}(u)$, where $u:\mathbb{R}^2=T_{x_0}\Sigma \rightarrow \mathbb{R}^{n-1}$, and by Lemma \ref{mean curvature equation} 
	\[
	\frac{a_{ij}(\nabla u_1,\ldots,\nabla u_{n-1})}{\sqrt{1+ |\nabla u_i}|^2}D_{ij} u_i=0.
	\]
	for every $i \in {1,\ldots, n-1}$.
	Lemma \ref{difference equation} implies that the difference $v=u-f_c$ satisfies a linear PDE of the following form:
	\[
	\frac{a_{ij}(\nabla u)}{\sqrt{1+ |\nabla u_k|^2}} D_{ij} v_k + \sum_{l=1}^{n-1}b_j^l(\nabla u,\nabla f_c) D_j v_l=0,
	\]
	for each $k=1,\ldots, n-1$. Note that  $v$ vanishes on $x_0$ and the order of vanishing is finite by Lemma \ref{unique continuation}.
	Therefore, $\mathcal{H}^1(v^{-1}(0)\cap |\nabla v|^{-1}(0)= 0$ by Lemma \ref{size nodal set}.
	This is a contradiction since $\Sigma$ and $\Sigma_c$ are tangent along $C(\Sigma)$ and $\text{dim}\,C(\Sigma)=1$. We conclude that  $v\equiv 0$ near $x_0$ and the theorem  follows from  standard analytic continuation property for minimal surfaces. 
\end{proof}

\begin{remark}\label{remark}
The same proof works for $2$-dimensional minimal surfaces   properly embedded in $\mathbb{R}^n$; the conclusion in this case is that such a surface $\Sigma^2$ satisfying (\ref{gap})  is either simply connected or the catenoid. In the special case $n=3$,  we can invoke the classification of properly embedded simply connected minimal surfaces by Meeks and Rosenberg \cite{MR} to conclude that $\Sigma$ is either the plane, the catenoid, or the helicoid. A simple computation shows that the helicoid does not satisfy (\ref{gap}).
\end{remark}

\end{document}